\documentclass[a4paper,11pt,reqno]{article}

\usepackage{a4wide}
\usepackage[latin1]{inputenc}

\usepackage{graphicx}
\usepackage{amsmath}
\usepackage{amsthm}
\usepackage{amssymb}
\usepackage{color}
\usepackage[normalem]{ulem} 
\usepackage {ulem} 

\usepackage{enumerate}

\usepackage{url}

\theoremstyle{plain}
\newcommand{\Tm}[1]{\ensuremath{T_{MRCA}^N( #1 )}}%
\newtheorem{theorem}{Theorem}[section]
\newtheorem{proposition}[theorem]{Proposition}
\newtheorem{lemma}[theorem]{Lemma}

\newtheorem{corollary}[theorem]{Corollary}

\newtheorem{definition}[theorem]{Definition}

\theoremstyle{definition}

\theoremstyle{remark}

\newcommand{\N}{\ensuremath{\mathbb{N}}}

\renewcommand{\P}{\ensuremath{\mathbb{P}}}

\newcommand{\be}{\begin{equation}}
\newcommand{\ee}{\end{equation}}

\title{Genealogy of a Wright Fisher model with strong seed bank component}

\author{Jochen Blath, Bjarki Eldon, Adri\'an Gonz\'alez Casanova, Noemi Kurt\footnote{Corresponding author. E-mail: {\tt kurt@math.tu-berlin.de}}\\
TU Berlin\footnote{Address: Institut f\"ur Mathematik, Sekr. MA 7-5, Stra\ss e des 17. Juni 136, 10623 Berlin, Germany}}
\date{\today}

\begin{document}

\maketitle
\begin{abstract}
We investigate the behaviour of the genealogy of a Wright-Fisher population model under the influence of a strong seed-bank effect. More precisely, we consider a simple seed-bank age distribution with two atoms, leading to either classical or long genealogical jumps (the latter modeling the effect of seed-dormancy). We assume that the length of these long jumps scales like a power $N^\beta$ of the original population size $N$, thus giving rise to a `strong' seed-bank effect. For a certain range of $\beta$, we prove that the ancestral process of a sample of $n$ individuals converges under a non-classical time-scaling to Kingman's $n-$coalescent. Further, for a wider range of parameters, we analyze the time to the most recent common ancestor of two individuals analytically and by simulation.

\end{abstract}

\section{Introduction}
We consider a generalization of the classical Wright Fisher model in the following sense: Consider a neutral, haploid population of fixed size $N$ that reproduces asexually in discrete generations indexed by the natural numbers. 
In contrast to the Wright-Fisher model, at each new generation, while most of the population stems from direct reproduction of individuals of the previous generation,  a few remaining individuals obtain their type from a parent having lived in the (possibly far) past. In the retrospective viewpoint, an individual alive in generation 0, instead of selecting its parent uniformly from the previous generation as in the classical  Wright Fisher model, uses some probability measure $\mu$ on $\N$ to 
sample the distance in generations to its parent,
and then picks its ancestor uniformly among the individuals at the sampled distance. The biological interpretation of this mechanism is that it allows old genes to become re-activated in a population after some time as the result of a \emph{seed bank effect}. Such an effect could be seen as an `evolutionary force' that may have to be taken into account for populations that produce dormant forms, such as plant seeds or bacterial spores. Some dormant forms may remain inactive for a long time, and, after becoming active again, potentially re-introduce old genetic material into the present population, thus increasing genetic variability \cite{Azotobacter, Levin1990, Tellier, Vitalis}. In the case where the dormancy is on a time scale which is non-negligible with respect to the population size, this is expected to lead to drastic changes in the genealogy of such a population. In the present paper we analyze a simple mathematical model which illustrates this effect.

\medskip

Informally, the model can be described as follows. Fix the population size $N\in\N$ and a 
probability measure $\mu$ on the natural numbers. This measure determines the generation of the immediate ancestor of an individual backward in time, meaning that an individual living at time $k\in\N$ has its immediate ancestor in generation $k+l$ with probability $\mu(l).$ 
Such a genealogical process, in the case where $\mu$ has \emph{finite} support independent of $N,$ was introduced and analyzed in \cite{KKL}, where it was shown that the genealogy converges, after classical rescaling by the population size, to a constant time change of Kingman's coalescent. 

\medskip 
In
\cite{BGKS}, the case of a stronger seed bank effect with unbounded measure $\mu$ was considered. More precisely, power law distributions of the form 
$$\mu_\alpha\big(\{n,n+1,...\}\big)=L(n)n^{-\alpha}$$
for some $\alpha>0$ and some slowly varying function $L(n)$ were investigated. Three regimes concerning the time to the most recent common ancestor were identified: If $\alpha>1,$ then
the expected time to the most recent common ancestor is of order $N,$ and the ancestral process converges to a constant time change of 
Kingman's coalescent under classical rescaling by the population size. For $1/2<\alpha<1,$ the time to the most recent common ancestor is finite almost surely, but
the expectation does not exist for any $N.$ If $\alpha<1/2,$ then there might be no common ancestor at all. The boundary cases $\alpha=1$
and $\alpha=1/2$ depend on the choice of $L(n).$

\medskip

In the present paper, we investigate a rather natural set up that was not considered in \cite{BGKS}. We chose $\mu$ to be of the following form:
For $N\in\N$ fixed, $\beta>0$ 
and $\varepsilon \in (0,1),$ let
\begin{equation}
\mu_N=(1-\varepsilon)\delta_1+\varepsilon\delta_{N^{\beta}}.
\end{equation}
This means that in each new generation, a proportion $(1-\varepsilon)$ of the total population obtains its genetic type from the previous
generation, whereas a fraction $\varepsilon$ of the population gets its type from generation $N^\beta$ in the past. The most important difference to previously studied models is the fact that the expected length of a genealogical jump is equal to $1+\varepsilon(N^{\beta}-1),$ and hence diverges as $N\to\infty.$ 
This puts us in a situation that is outside the scope of \cite{KKL} or \cite{BGKS}.

\medskip

Our main result shows that for $0<\beta<1/4,$ after rescaling time by $\varepsilon^2 N^{2\beta+1},$ the ancestral process of a sample from such a population converges to Kingman's $n-$coalescent, showing that the relevant time scale in $N$ is indeed much larger than in the Wright-Fisher model, namely $N^{2\beta+1}$ as opposed to $N.$ Moreover, we show that for any $\beta>0$ the time to the most recent common ancestor of two individuals is always of an order that is strictly greater than $N,$ and we provide some simulations that support the conjecture that $N^{2\beta+1}$ is the relevant time scale also for (at least some) $\beta>1/4.$

\section{Model and main results}
The formal construction of our model follows \cite{KKL, BGKS}. Fix $\beta>0, \varepsilon\in(0,1).$ For each $N\in\N$ let

\begin{equation}\label{eq:mu_N}
\mu_N:=(1-\varepsilon)\delta_1+\varepsilon\delta_{N^{\beta}}.
\end{equation}

In order to simplify notation, we will assume throughout this paper that $N^\beta$ is a natural number, otherwise imagine it replaced by $\lfloor N^\beta\rfloor.$ 
We call the probability measure $\mu_N$ the \emph{seed bank age distribution}. Fix once and for all a reference generation 0, from which time in discrete generations runs backwards. Fix a sample size $m\geq 2$ and a sampling measure $\gamma$ for the generations of the original sample on the integers $\N.$ We will usually assume that $\gamma$ has finite support (independent of $N$), an important example being $\gamma=\delta_0.$ The ancestral lineages of $m$ sampled individuals indexed by $i\in\{1,...,m\}$ in the seed bank process, who lived in generations sampled according to $\gamma$ with respect to reference time 0, are constructed as follows.  For each $i\in\{1,...,m\},$ let $(S^{(i)}_n)_{n\in\N}$ be a Markov chain independent of $\{(S_n^{(j)})_n, j\neq i\},$ whose state space is the non-negative integers $\N_0,$ with $S_0^{(i)}\sim\gamma,$ and homogeneous transition probabilities 

\[ \P\big(S_1^{(i)}=k'\,\big|\, S_0^{(i)}=k\big)=\mu_N(k'-k),\; 0\leq k<k', \; i=1,...,m.\]


The interpretation is that $S_0^{(i)}$ represents the generation of individual $i,$ and $S^{(i)}_1$ the generation of its parent (backward in time), and so on. The set $\{S^{(i)}_0, S^{(i)}_1,...\}\subset \N_0$ is thus the set of generations of all ancestors of individual $i,$ including the individual itself. 

\medskip

In order to construct the ancestral process of several individuals, we introduce interaction between ancestral lines as follows. Within the population of size $N,$ in any fixed generation $k,$ the individuals are labeled from 1 to $N.$ Let $(U^{(i)}_n)_{n\in\N}, i\in\{1,...,m\}$ denote $m$ independent families of independent random variables distributed uniformly on $\{1,...,N\},$ independent of $\{(S_n^{(i)}), i=1,...,m\}.$ We think of $U^{(i)}_{S^{(i)}_n}$ as the label within the population of size $N$ of the $n$th ancestor of individual $i.$ This means that the label of each ancestor of each individual is picked uniformly at random in each generation that the ancestral line of this individual visits, exactly as it is done in the Wright-Fisher model. The difference is that an ancestral line in the Wright-Fisher model visits every generation, while in our seed bank model it does not. Note that of course all the random variables introduced up to now depend on the population size $N.$

\medskip

The \emph{time to the most recent common ancestor} of two individuals $i$ and $j$, denoted by $T_{MRCA}(2),$ is defined as

\be T_{MRCA}(2):=\inf\big\{k >0: \exists n,m\in\N, k=S^{(i)}_n=S^{(j)}_m \mbox{ and } U^{(i)}_{k}=U^{(j)}_{k}\big\}.
\ee

In other words, $T_{MRCA}(2)$ is the first generation back in time (counted from 0 on) in which two randomly sampled individuals (`initial' generations sampled according to $\gamma)$ have an ancestor, and both ancestors have the same label $U,$ hence, it is indeed the first generation back in time that $i$ and $j$ have the \emph{same} ancestor.

\medskip

It should be clear how to generalize this construction to lead to a full ancestral process of $m\geq 2$ individuals: Construct the process $(S^{(i)}_n, U^{(i)}_n)_{n\in\N}$ independently for each individual, and couple the lines of individual $i$ and individual $j$ at the time of their most recent common ancestor by letting them evolve together from this time onward, as represented in Figure 1.

\begin{figure}[h!]
\begin{center}
\includegraphics[scale=0.4]{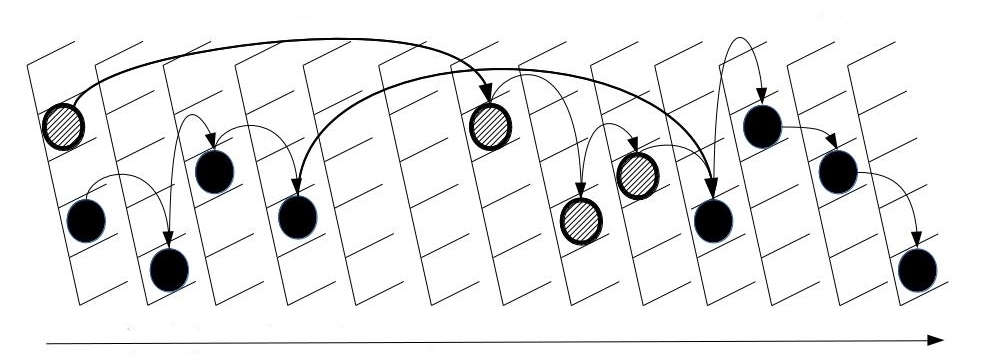}
\caption{Coalescing ancestral lines of two individuals. Backward time is running from left to right. The slots within each generation indicate the different individuals.}
\end{center}
\end{figure}

A precise construction is given in the following way: Let 

\be T_1:=\inf\big\{k>0: \exists n,l\in\N, i\neq j\in\{1,...,m\}: k=S_n^{(i)}=S_l^{(j)}, U_k^{(i)}=U_k^{(j)}\big\},\ee
be the time of the first coalescence of two (or more) lines, and let the set of individuals whose lines participate in a coalescence at time $T_1$ be denoted by

\be I_1:=\big\{i\in\{1,...,m\}: \exists n,l\in\N, j\neq i: T_1=S_n^{(i)}=S_l^{(j)}, U_{T_1}^{(i)}=U_{T_1}^{(j)}\big\}.\ee
$I_1$ can be further divided into (possibly empty) pairwise disjoint sets

\be I_1^p:=\big\{i\in I_1: U_{T_1}^{(i)}=p\big\},\quad p=1,...,N.\ee
Note that by construction there is at least one $p$ such that $I_1^p$ is non-empty, which actually means that any such $I_1^p$ contains at least two elements. Let 
\[i^p_1:=\min I^p_1,\]
and let 
\[J_1:=\bigcup_{p: I_1^p\neq \emptyset}\{i_1^p\}.\]
After time $T_1$ we discard all $S^{(j)}$ for $j\in I^p_1, j\neq i^p_1,$ and only keep $S^{(i^p_1)}$ for every $p=1,...,N.$ We interpret this as merging the ancestral lineages of all individuals from $I_1^p$ into one lineage at time $T_1,$ separately for every $p$ with $I_1^p\neq \emptyset.$ In case there are several non-empty $I_1^p,$ we observe simultaneous mergers. For $r\geq 2$ we define now recursively

\be T_r:=\inf\big\{k>T_{r-1}: \exists n,l\in\N,\exists i, j\in I^c_{r-1}\cup J_{r-1}, i\neq j: k=S_n^{(i)}=S_l^{(j)}, U_k^{(i)}=U_k^{(j)}\big\},\ee
and
\be I_r:=\big\{i\in I^c_{r-1}\cup J_{r-1}: \exists n,l\in\N,\exists j\neq i, j\in I^c_{r-1}\cup J_{r-1}: T_r=S_n^{(i)}=S_l^{(j)}, U_{T_r}^{(i)}=U_{T_r}^{(j)}\big\},\ee
and similarly $I_r^p:=\{i\in I_r: U_{T_r}^{(i)}=p\}, i^p_r=\min I^p_r,  p=1,...,N,$ and $J_r= \cup_{p: I_r^p\neq \emptyset}\{i_r^p\}.$ We stop the recursive construction as soon as $I_r^c=\emptyset,$ which happens after finitely many $r.$ Now we can finally define the main object of interest of this paper.

\begin{definition}\label{def:ancestral} Fix $N\in\N, \beta>0,$ and $\varepsilon >0.$ Fix $m\ll\N$ and an initial distribution $\gamma$ on $\N_0.$ Define a partition valued process $(A_k^N)_{k\in\N_0},$ starting with $A_0^N=\big\{\{1\},...,\{m\}\big\}$, by setting $A^N_k=A^N_{k-1}$ if $k\notin\{T_1, T_2,..\},$ and constructing the $A^N_{T_r}, r=1,2,...$ in the following way: For each $p\in\{1,...,N\}$ such that $I_r^p\neq \emptyset,$ the blocks of $A^N_{T_r-1}$ that contain at least one element of $I^p_r,$ are merged. Such merging is done separately for every $p$ with $I_r^p\neq \emptyset,$ and the other blocks are left unchanged. The resulting process $(A^N_k)_{k\in\N}$ is called the \emph{ancestral process} of $m$ individuals in the Wright-Fisher model with seed bank age distribution $\mu_N$ and initial distribution $\gamma.$ The time to the most recent common ancestor of the $m$ individuals is defined as
\be \label{eq:TMRCAN}  T^N_{MRCA}(m):=\inf\big\{k\in\N: A^N_k=\{1,...,m\}\big\}.\ee
\end{definition}

 It is important to note that $A^N$ is not a Markov process: The probability of a coalescence at time $k$ depends on more than just the configuration $A^N_{k-1}.$ In fact, it depends on the values $\max\{S^{(i)}_n: S^{(i)}_n\leq k-1\}, i=1,...,m,$ that is, on the generation of the last ancestor of each individual before generation $k.$

\medskip

An equivalent construction of $(A^N_k)$ in terms of renewal processes was given in \cite{BGKS}. In the present paper, we denote by 
$P_\gamma$ the law of $(S_n^{(1)})$, indicating the initial distribution of the generations of the individuals. We write $P_{\otimes\gamma^m}$ for the law of the process $(A_k^N)$ if the generation of each of the $m$ sampled individuals is chosen independently according to $\gamma.$ We abbreviate by slight abuse of notation both $P_{\delta_0}$ and $P_{\otimes \delta_0^m}$ by $P_0.$ In the main result below we will assume $\gamma$ to have finite support independent of $N.$ In this case, the fact that the individuals may be sampled from different generations will become negligible after rescaling time appropriately. However, for the construction of $(A_k^N)$ this assumption is not necessary.

\medskip

The aim of this paper is to understand the non-trivial scaling limit of $(A^N_k)$ and the corresponding time-scaling as $N\to\infty.$ We write $\mathcal{M}_1(\N_0)$ for the probability measures on $\N_0.$ We are now ready to state our main result.

\begin{theorem}\label{thm:kingman}
Let $0<\beta<1/4.$ For all $m>0$ and $\gamma\in\mathcal M_1(\N_0)$ with finite support, the process $(A^N_{\lfloor\varepsilon ^2 N^{1+2\beta} t\rfloor})_{t\geq 0}$ converges weakly as $N\to\infty$ on the Skorohod space of c\`adl\`ag paths to Kingman's $m-$coalescent.
\end{theorem}

This result should be compared to the classical result for the ancestral process of the Wright-Fisher model: In that case, convergence to Kingman's coalescent occurs on the time scale $N$ corresponding to the population size. Our result shows that the seed bank effect drastically changes the time scale on which coalescences occur. The intuitive reason for this is that each generation is visited with probability $\approx\varepsilon^{-1}N^{-\beta}$ by the ancestral line of any given individual, and whenever two ancestral lines visit the same generation, they merge with probability $1/N.$ Making this intuition precise requires some work, which is carried out in the next two sections of this paper. As a direct consequence of this theorem, the time to the most recent common ancestor of $m$ individuals is of strictly larger order than in the Wright-Fisher model.

\begin{corollary}%
\label{cor:scaledxptime}%
Let $0<\beta<1/4.$ For all $m>0$ and $\gamma\in\mathcal M_1(\N_0)$ with finite support,
\be \lim_{N\to\infty} \frac{E_{\otimes\gamma^m} [T^N_{MRCA}(m)]}{\varepsilon^2 N^{1+2\beta}}=2\Big(1-\frac{1}{m}\Big).\ee
\end{corollary}

\begin{proof}
Recalling that for Kingman's $m-$coalescent the expected time to the most recent common ancestor is $2(1-1/m),$ this result follows from Theorem \ref{thm:kingman}.
\end{proof}

As $\beta$ increases above $1/4,$ one expects the seed bank effect to become even more pronounced. However, our methods of proof, relying on mixing time arguments, don't yet allow us to extend the result to arbitrary $\beta>0.$ It is nevertheless not difficult to see that in any case, the expected time to the most recent common ancestor is of course of order greater than $N$ for any choice of $\beta>0,$ and that $T^N_{MRCA}$ is of order at least $N^{2\beta+1}$ with probability tending to 1 if $0<\beta<1/3.$

\begin{proposition}\label{thm:mrca} Fix $m\in\N,$ and $\gamma\in\mathcal M_1(\N_0)$ with finite support.
\begin{itemize}
\item[(i)] Let $\beta>0.$ For all $N\in\N$ large enough,
\be
E_{\otimes\gamma^m}\big[T^N_{MRCA}(m)\big]\geq \varepsilon  N^{1+\beta}\vee N.
\ee
\item[(ii)]Let $0<\beta<1/3.$ For all $\delta>0,$
\be \lim_{N\to\infty}P_{\otimes \gamma^m}\big(T^N_{MRCA}(m)<\varepsilon^2 N^{1+2\beta-\delta}\big)=0.\ee
\end{itemize}
\end{proposition}

We will prove these results in Section 4. In the next section, we give an alternative construction of the model in terms of an auxiliary 
Markov process that will be useful in the proof.

\medskip

In Section 5 we present some simulations of $T^N_{MRCA}(2)$ for certain choices of $\varepsilon$ and $\beta,$ where both lines are sampled from the same generation. The simulations show that also for $\beta=1/3$ and $\beta=1/2,$ the empirical distribution of $T^N_{MRCA}(2),$ scaled by a factor $\varepsilon^2 N^{2\beta+1},$ exhibits a very good fit to an exponential random variable with parameter 1.

\section{Construction of auxiliary processes}\label{sect:urn}

In \cite{KKL} and \cite{BGKS}, an auxiliary urn Markov process plays a crucial role. We present this process now in a set-up that 
is useful for this paper, and derive some properties that will serve us in the proof of our main results.

\medskip

Fix $N, \varepsilon, \beta$ and $\mu_N$ as in the previous section. Fix a probability measure $\gamma$ on $\N$ with finite support, and assume that $N$ is large enough such that $\mbox{supp}(\gamma)\subseteq \{0,...,N^{\beta}\}.$ Let $\P_{\gamma}$ be the law of a Markov chain $(X_k)_{k\in\N_0}$ on $\{0,...,N^{\beta}-1\}$ with initial distribution $\gamma$ that moves according to the following rules: For $k\geq 0,$ depending on the current state $X_k,$ we have transitions

\be \label{eq:transitions-urn} \quad X_k\mapsto \begin{cases}
                  0&\mbox{with probability }(1-\varepsilon)1_{\{X_k=0\}},\\
                  N^\beta-1&\mbox{with probability }\varepsilon 1_{\{X_k=0\}},\\
                  X_{k-1}&\mbox{with probability }1_{\{X_k\in \{1,...,N^\beta-1\}\}}.
                 \end{cases}\ee
As in \cite{KKL}, we call this process an \emph{urn process}, because we think of $X_k$ as the position (urn) of a ball that is moved among $N^\beta$ urns. Figure 2 shows the possible jumps of $(X_k).$

\begin{figure}[h!]
\begin{center}
\includegraphics[scale=0.4]{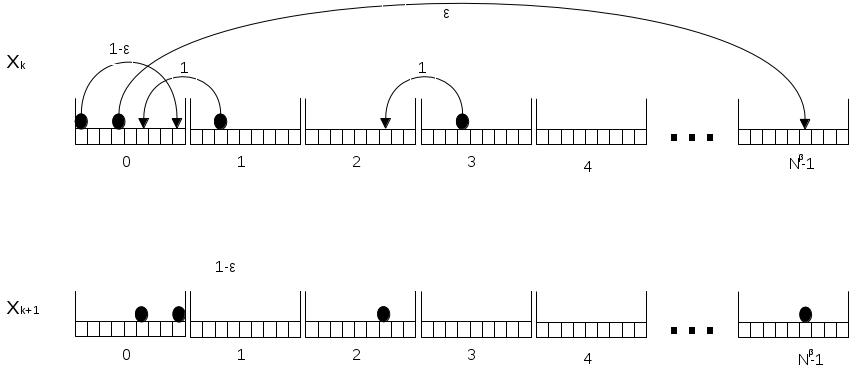}
\caption{The possible jumps of $X_k:$ From urn 0 it stays in 0 with probability $1-\varepsilon,$ jumps from urn 0 to urn $N^\beta-1$ with probability $\varepsilon,$ and deterministically downwards to the previous urn from any other urn.}
\end{center}
\end{figure}

How does this new process connect to the original processes $(S^{(i)}_k)$ resp. our ancestral process $(A_k^N)$? We can 
couple $X$ and $S^{(i)}$ such that the successive times that $X$ visits urn 0 are exactly the successive values visited by the process 
$S^{(i)},$ that is, the generations in which individual $i$ has an ancestor. This coupling is achieved as follows: Define 
\be \label{eq:Mproc}M_0:=\inf\{k\geq 0: X_k=0\},\quad M_{n}=\inf\{k> M_{n-1}: X_k=0\},\; n\geq 1.\ee
Then we have

\begin{lemma}\label{lem:urn-ancestral}
Let $(X_k)$ be the above urn process with initial distribution $\gamma$ on $\{0,..., N^{\beta}-1\},$ and let $(M_n)_{n\in\N_0}$ be defined as in \eqref{eq:Mproc}. Then the process $(\tilde{S}_n)_{n\in\N_0}$ defined by 
\be \label{eq:urn-ancestral}\tilde{S_0}:=M_0,\quad \tilde{S}_n:=\tilde{S}_{n-1}+M_n\ee
has the same distribution as $(S_n^{(i)})$ started in $\gamma,$ and for all $k\in\N,$

\be \label{eq:prob_urn1} P_{\gamma}(\exists n:S_n^{(i)}=k)=\P_{\gamma}(X_k=0).\ee
\end{lemma}

\begin{proof}
Immediate by construction.
\end{proof}

Note that, conversely, given $S_n=\tilde{S}_n, n\in\N_0,$ equation \eqref{eq:urn-ancestral} uniquely determines $M_n, n\in\N_0.$ If $N$ is large enough such that the distribution $\gamma$ of $S_0$ satisfies supp$(\gamma)\subseteq\{0,..., N^\beta-1\},$ then it is a possible initial distribution for $(X_k).$ Further, since the urn  process $(X_k)$ is determined by the successive times it visits urn 1, Lemma \ref{lem:urn-ancestral} yields a one-to-one correspondence between ancestral lines in the seed bank process, and the above urn process.  As in the construction of $(A_k^N),$ for the process $(X_k)$ we need not assume that $\gamma$ has finite support independent of $N,$ indeed it will actually be useful to allow for initial distributions $\gamma=\gamma_N$ with $\mbox{supp}(\gamma_N)=\{0,...,N^\beta-1\}$ depending on $N,$ for each $N\in\N.$ 

\medskip

Let $E_{\mu_N}$ denote the expectation of $\mu_N,$ i.e.\ $E_{\mu_N}=1+\varepsilon(N^\beta-1).$

\begin{lemma}\label{lem:stat}
The probability measure $\nu_N$ on $\{0,...,N^\beta-1\}$ defined by
 \begin{equation}\label{def:nu}\nu_N(k):=\frac{\mu_N\big(\{k+1,...\}\big)}{E_{\mu_N}}, \quad k=0,...,N^\beta-1,\end{equation}
 is the unique stationary distribution of the urn process $(X_k).$
\end{lemma}
\begin{proof}
This is a special case of Lemma 1 in \cite{KKL} (note that there the urn process is shifted by one, i.e.\ takes values in $\{1,...,m\}$ for some $m$ corresponding to our $N^\beta,$ instead of $\{0,...,m-1\}).$
\end{proof}

In view of Lemma \ref{lem:urn-ancestral}, an important quantity in this paper will be the probability that $X_k=0,$ which under stationarity is equal to

\be \nu_N(0)=\frac{1}{1+\varepsilon(N^\beta-1)}\sim \frac{1}{\varepsilon N^\beta} \;\mbox{ as }N\to\infty.\ee

The following lemma will be crucial in the proof of our main results. By $\|\cdot \|_{TV}$ we denote the total variation distance, which for two probability measures $\mu,\nu$ on a measureable space $(\Omega, \mathcal F)$ is given by
\[\|\mu-\nu\|_{TV}=\sup_{A\in\mathcal F}|\mu(A)-\nu(A)|.\]

\begin{lemma}\label{lem:total_variation}
Let $(X_k)$ be the urn process and $\nu_N$ its stationary distribution. For $\beta>0$ let $\mathcal{P}_{N^\beta}$ denote the set of probability measures on $\{0,...,N^\beta-1\}.$
\begin{itemize} 
\item[(i)]For all $\lambda>3\beta>0,$ there exist $\delta>0$ and $ N_0\in\N,$ such that for all $N\geq N_0$
\[\sup_{\mu\in\mathcal{P}_{N^\beta}}\|\P_{\mu}(X_{N^{\lambda}}\in \cdot\,)-\nu_N\|_{TV}\leq e^{-N^\delta}.\]
\item[(ii)] Let $\tau=\tau(N)$ be a geometric random variable with parameter $1/N$ independent of $(X_k).$ If $0<\beta<1/4,$ then there exists $\delta>0$ and $ N_0\in\N$ such that for all $N\geq N_0$
\[\sup_{\mu\in\mathcal{P}_{N^\beta}}\|\P_{\mu}(X_\tau\in \cdot\,)-\nu_N\|_{TV}\leq N^{-(\beta+\delta)}.\]
\end{itemize}
\end{lemma}

\begin{proof}
Fix $\mu\in\mathcal{P}_{N^\beta}.$ Let $(Z_n)_{n\in\N_0}$ be a realization of the urn process started in the invariant distribution $\nu_N$ independent of $(X_n)_{n\in\N_0}.$ We couple $(X_n)$ and $(Z_n)$ by a Doeblin coupling in the following way: Let $\sigma_0:=\inf\{n\in \N_0: X_n=Z_n\}.$ Define 
\[\tilde{X}_n:=\begin{cases} X_n & \mbox{if }n\leq \sigma_0,\\
Z_n& \mbox{if }n>\sigma_0.\end{cases}
\]
Write $\P:=\P_{\mu\otimes\nu_N}.$ Then $\P(\tilde{X}_n=k)=\P_\mu(X_n=k)$ for all $n\in\N_0, k\in \{0,...,N^{\beta}-1\}.$
By Proposition 4.7 of \cite{LevinPeresWilmer}, we have
\be\label{eq:totvar} \|\P_\mu(X_n\in \cdot\,)-\nu_N\|_{TV}\leq \P(\tilde{X}_n\neq Z_n)=\P(\sigma_0>n).\ee
Our aim is therefore to bound $\P(\sigma_0>n).$ To this end we consider the difference of the two process at particular times. Define $m_0:=\inf\{n\geq 0:X_n=0\}, l_0:=\inf\{n\geq 0:Z_n=0\},$ and let recursively, for $i\geq 1,$

\[m_i:=\inf\{n>m_{i-1}:X_n=0, X_{n-1}=1\}\]
and
\[l_i:=\inf\{n>l_{i-1}:Z_n=0, Z_{n-1}=1\}.\]
Note that for all $i\geq 0$ we have $Z_{m_i}-X_{m_i}\geq 0$ and $X_{l_i}-Z_{l_i}\geq 0.$ Without loss of generality we can assume that $Z_0-X_0>0,$ which implies $m_0<l_0.$ Since the difference of the two processes remains constant as long as none of the two processes is in urn 0, we see that 
\be \sigma_0\in \{m_i:i\geq 2\}\cup \{l_i: i\geq 1\},\ee
i.e. the coupling always happens in urn 0, and it happens if either process $(Z_n)$ jumps from 1 to 0 while $(X_n)$ is in 0 or vice versa. 

Define for $i\geq 0$
\[ V_i:=\big|\big\{n\in\{m_i,...,m_{i+1}-1\}: X_n=0\big\}\big|,\]
and
\[W_i:=\big|\big\{n\in\{l_i,...,l_{i+1}-1\}: Z_n=0\big\}\big|,\]
the number of visits in urn 0 of either of the process during one `cycle' (note that between $m_i$ and $m_{i+1}$ the process $(X_k)$ has exactly one jump of lenght $N^\beta.$ By construction, $(V_i)_{i\geq 0}$ and $(W_i)_{i\geq 0}$ are independent sequences of iid geometric random variables with parameter $\varepsilon,$ and 
\be \label{eq:increments1}(Z_{m_i}-X_{m_i})-(Z_{m_{i-1}}-X_{m_{i-1}})=W_{i-1}-V_{i-1},\ee
\be \label{eq:increments2}
(X_{l_i}-Z_{l_i})-(X_{l_{i-1}}-Z_{l_{i-1}})=V_{i-1}-W_{i-1},\ee
$i\geq 1.$ Moreover we note that
\be \label{eq:distanceprocess}m_{i+1}-m_i=V_i+N^\beta,\quad l_{i+1}-l_i=W_i+N^\beta.\ee

The random sequence $(\sum_{i=0}^k(V_i-W_i))_{k\geq 0}$ is a random walk with centered increments whose variance (depending on $\varepsilon$ but not on $N$) is finite. Moreover, $\sigma_0$ can be controlled by the first time this random walk exits the set $\{-N^\beta+1,...,N^\beta-1\},$ since this event corresponds to either $(Z_n)$ `catching up' with $(X_n),$ or vice versa. More precisely, defining

\[ R:=\inf\big\{k\geq 0: |\sum_{i=0}^k (V_i-W_i)|\geq N^\beta\big\},\]
we see from \eqref{eq:increments1} and \eqref{eq:increments2} that
\be \label{eq:sigma_R} \sigma_0\leq m_R.\ee

Equation \eqref{eq:sigma_R} implies that for any $\lambda>0,$

\be \label{eq1}
\begin{split}
\P\big(\sigma_0>N^\lambda\big)&\leq  \P\big(\sum_{i=1}^R(m_i-m_{i-1})>N^\lambda\big)\\
&=1-\P\big(\sum_{i=1}^R(m_i-m_{i-1})\leq N^\lambda\big)\\ 
&\leq 1-\P\big(\{R<\frac{1}{2}N^{\lambda-\beta}\}\cap\{m_i-m_{i-1}
\leq 2N^\beta \;\forall i=1,...,\frac{1}{2}N^{\lambda-\beta}\}\big)\\
&\leq \P\big(\{R>\frac{1}{2}N^{\lambda-\beta}\}\cup\{\exists 1\leq i\leq \frac{1}{2}N^{\lambda-\beta}: m_i-m_{i-1}>2N^\beta\}\big)\\
&\leq \P\big(R>\frac{1}{2}N^{\lambda-\beta}\big)+\P\big(\exists 1\leq i\leq N^{\lambda-\beta}: m_i-m_{i-1}>2N^\beta\big).
\end{split}
\ee
To control the first term on the rhs, we use classical bounds on the exit time from an interval of symmetric random walks with finite variance, see e.g. Theorem 23.2 of \cite{Spitzer}. This provides that for every $\delta'>0$ there exists $\delta>0$ such that

\be \P\big(R>N^{2\beta+\delta'}\big)\leq e^{-N^\delta}.\ee
For $\lambda>3\beta,$ we can choose $\delta'>0$ such that $2\beta+\delta'<\lambda-\beta,$ hence we find the bound
\be\label{eq2} \P\big(R>\frac{1}{2}N^{\lambda-\beta}\big)\leq e^{-N^\delta}.\ee
To bound the second term in \eqref{eq1}, by \eqref{eq:distanceprocess} and a union bound we find, for $N$ large enough,

\be\label{eq3}\begin{split} \P\big(\exists 1\leq i\leq N^{\lambda-\beta}: m_i-m_{i-1}>2N^\beta\big)&\leq N^{\lambda-\beta}\P\big(V_1>N^\beta\big)\\
&= N^{\lambda-\beta}(1-\varepsilon)^{N^\beta}\leq N^{\lambda-\beta}e^{-\varepsilon N^\beta}\leq e^{-N^{\beta/2}}.\end{split}\ee
In view of \eqref{eq:totvar}, together the bounds \eqref{eq1}, \eqref{eq2} and \eqref{eq3} prove (i). For (ii), recall (see e.g.\cite{LevinPeresWilmer}, Chapter 4), that for any $k\geq l$ and $\mu\in\mathcal P_{N^\beta},$
\[\|\P_\mu(X_k\in\cdot\,)-\nu_N\|_{TV}\leq \|\P_\mu(X_l\in\cdot\,)-\nu_N\|_{TV}.\]
Using this, we see that for any $0<\lambda<1,$
\be \label{eq4}\begin{split}\|P_\mu (X_\tau&\in \cdot\,)-\nu_N\|_{TV} \\
\leq&\|\P_\mu(X_{N^{\lambda}}\in \cdot\,)-\nu_N\|_{TV}\cdot \P(\tau\geq N^{\lambda})+\|\P_\mu(X_{0}\in \cdot\,)-\nu_N\|_{TV}\cdot \P(\tau< N^{\lambda})\\
\leq & \|\P_\mu(X_{N^{\lambda}}\in \cdot\,)-\nu_N\|_{TV}+ \P(\tau< N^{\lambda}).\end{split}
\ee
If $3\beta<\lambda<1,$ we can bound the first term using (i), and second term 
by choosing $N$ large enough such that the Bernoulli inequality yields
\be \label{eq:bound2} \P(\tau<N^{\lambda})=1-\big(1-\frac{1}{N}\big)^{N^{\lambda}}\leq N^{\lambda -1}.\ee

Since we assumed $\beta<1/4,$ there exists $\delta>0$ such that we can chose $\lambda$ in a way that that $0<3\beta<\lambda<1-\beta-\delta<1,$ so that 
\[ \P(\tau<N^{\lambda})\leq N^{\lambda -1}\leq N^{-(\beta+\delta)}.\]
Plugging this together with \eqref{eq4} into (i) completes the proof.
\end{proof}

To put this result into a context, recall that the \emph{mixing time} of a Markov chain $(Y_n)$ with invariant distribution $\nu$ is often defined as 

\[\tau_{mix}:=\inf\big\{n>0: \sup_\mu\|P_\mu(Y_n=\cdot\;)-\nu\|_{TV}\big\}\leq \frac{1}{4}\]
(cf.\ for example \cite{LevinPeresWilmer}). It thus follows immediately from the previous lemma that for the urn process of the seed bank model,

\be \tau_{mix}\leq N^{3\beta+\delta}\ee
for all $\delta>0.$ This observation allows us to justify that, if $\beta<1/4,$ each generation is visited with probability $\approx \varepsilon^{-1}N^{-\beta}\sim\nu_N(0).$ Indeed, if $\beta$ is small enough, then for $N$ large the urn process is close to stationarity before the first coalescence, or actually, as we will see in the next section, before the first \emph{attempt} to coalesce. Note that Lemma \ref{lem:total_variation} is crucial to achieve this, and to prove this lemma we needed $\beta<1/4.$

\section{Proof of the main results}\label{sect:proof}

We are now going to use the urn process from the previous section in order to prove our main results. The crucial step is to calculate the time until the coalescence of two lines in terms of this urn process. 

\medskip

Let $(X_k)_{k\geq 0}$ and $(Y_k)_{k\geq 0}$ be two independent copies of the urn process, corresponding in the sense of Lemma \ref{lem:urn-ancestral} to the ancestral lines of two individuals $i$ and $j.$  Let $\tau_0=0$ and let $\tau_k,k\geq 1,$ be such that $(\tau_k-\tau_{k-1})_{k\geq 1}$ is
a sequence of independent geometric random variables with parameter $1/N,$ independent of $X$ and $Y.$ Let 
\be \kappa:=\inf\big\{k\geq 1: X_{\tau_k}=Y_{\tau_k}=0\big\}.\ee

\begin{lemma}\label{lem:change}
Assume that $X,Y$ independently follow the urn dynamics started from initial distribution $\gamma$ with support in $\{0,...,N^\beta-1\}.$ Then for $S^{(1)}, S^{(2)}$ started from $\gamma$,
\be T_{MRCA}(2)\stackrel{d}{=}\tau_\kappa.\ee
\end{lemma}

\begin{proof}
Let two sequences $(U_k^{(1)})_{k\geq 1}$ and $(U_k^{(2)})_{k\geq 1}$ of independent uniform random variables  on $\{1,...,N\},$ independent of $S, X$ and $Y$ be defined as in Section 2. Without loss of generality, due to independence, we can assume that $\tau_1=\inf\big\{k\geq 0: U_k^{(1)}=U_k^{(2)}\big\}$ and $\tau_k=\inf\big\{k>\tau_{k-1}: U_k^{(1)}=U_k^{(2)}\big\}, k\geq 2.$ Hence 
\[T_{MRCA}(2)=\inf\big\{\tau_k>0: \exists n,m\in\N: \tau_k=S_n^{(1)}=S_m^{(2)}\big\}.\]
Now the claim follows from equation \eqref{eq:prob_urn1} and the independence of the two lines.
\end{proof}

\paragraph*{Proof of Theorem \ref{thm:kingman}.}  Fix $0<\beta<1/4$ and $m\geq 2.$ Clearly for all $N\in\N$ the process $(A^N_k)_{k\in\N}$ is exchangeable. In particular, its dynamics does not depend on the respective sizes or order of blocks, or their elements, but only on the number of blocks. Therefore it is sufficient to consider the block-counting process $(|A_k^N|)_{k\in\N_0},$ which is a pure death process started from a fixed $m\in\N$ whose distribution is uniquely determined by the sequence of times at which a coalescence happens, and the specification of the type of coalescence at each of these times. Hence, to prove our main result, it is sufficient to show that the inter-coalescence times of $(A_k^N)$ are asymptotically (as $N\to\infty$) independent and exponentially distributed with the right parameters, and that multiple and simultaneous mergers are negligible. 

\medskip

In order to determine the distribution of the inter-coalescence times, we use Lemma \ref{lem:change}, which tells us that the time until two given lines coalesce -- ignoring the influence of the other lines -- is distributed as $\tau_\kappa.$ Since the differences $(\tau_k-\tau_{k-1})$ are geometric, we only need to determine the distribution of $\kappa.$ By Lemma \ref{lem:total_variation} (ii) there exists $\delta>0$ and $r_N$ such that $|r_N|\leq N^{-(\beta+\delta)},$ and

\[\P_{\gamma}(X_{\tau_1}=0)=\nu_N(0)+ r_N.\]
Since $\nu_N(0)\sim\varepsilon^{-1}N^{-\beta}$ as $N\to\infty$, we have $\lim_{N\to\infty}r_N/\nu_N(0)=0$ (note that $\varepsilon$ is independent of $N$), and hence
\be \P_{\gamma\otimes \gamma} (\kappa=1)=\P_{\gamma\otimes \gamma}(X_{\tau_1}=Y_{\tau_1}=0)=\P_\gamma(X_{\tau_1}=0) \P_\gamma(Y_{\tau_1}=0)
=\nu_N(0)^2(1+o(1)).\ee
Since the result of Lemma \ref{lem:total_variation} is uniform in the initial 
distribution $\mu,$ we see that for all $k\in\N,$ by the Markov property of $(X_k)$ and $(Y_k),$

\be \P_{\gamma\otimes \gamma}(\kappa=k\,|\, \kappa>k-1)=\nu_N(0)^2(1+o(1)),\ee
with the $o(1)-$notation independent of $k.$ 

This shows that $\kappa$ is asymptotically geometric distributed with parameter $\nu_N(0)^2$. More precisely, for any $c>0$ and geometric random variables 
$G_1$ with parameter $\nu_N(0)^2(1+c)$ and $G_2$ with parameter $\nu_N(0)^2(1-c)$ there exists $N_0\in\N$ such that for all $N\geq N_0$ we have
that $G_2\leq \kappa \leq G_1$ stochastically. Since
\be
T^N_{MRCA}(2)\stackrel{d}{=}\tau_{\kappa}=\sum_{k=1}^{\kappa}(\tau_k-\tau_{k-1}),
\ee
we get 
\be\sum_{i=1}^{G_2}(\tau_k-\tau_{k-1})
\stackrel{d}{\leq} T^N_{MRCA}(2)\stackrel{d}{\leq} \sum_{k=1}^{G_1}(\tau_k-\tau_{k-1}),\ee
$\stackrel{d}{\leq}$ indicating stochastic dominance.
An easy calculation using moment generating functions shows that a sum of a geometric number of geometric random variables is 
again geometrically distributed (with a parameter which is the product of the two parameters). Hence 
\[\sum_{k=1}^{G_1}(\tau_k-\tau_{k-1})\sim\mbox{Geo}(N^{-1}\nu_N(0)^2(1+c)),\]
 and similarly for the sum up to $G_2.$ From this we conclude that for any $t>0,$
\begin{equation}\begin{split}
\lim_{N\to \infty}P_{\gamma\otimes\gamma}\big(T^N_{MRCA}(2)>\lfloor \varepsilon^2 N^{2\beta+1}\rfloor t\big)\geq \lim_{N\to\infty}\Big(1-\frac{1}{N}\nu_N(0)^2(1+c)\Big)^{\lfloor \varepsilon^2 N^{2\beta+1}\rfloor t}
=e^{-t(1+c)},
\end{split}
\end{equation}
and similarly

\be\lim_{N\to \infty}P_{\gamma\otimes\gamma}\big(T^N_{MRCA}(2)>\lfloor \varepsilon^2 N^{2\beta+1}\rfloor t\big)\leq e^{-t(1-c)}.\ee
Since $c>0$ was arbitrary, we see that asymptotically as $N\to\infty,$ the time until two lines coalesce converges in distribution to the time of coalescence in Kingman's 2-coalescent. Further, if we consider $m\geq 2$ lineages, the first coalescence is dominated by resp. dominates the minimum of $\binom{m}{2}$ independent geometric variables with parameters $N^{-1}\nu_N(0)(1\pm c)$ each. Therefore, denoting by $T_{coal}(m)$ the time until the first two lines coalesce, we have in a similar fashion

\be 
\begin{split}
\lim_{N\to \infty}P_{\otimes\gamma^m}\big(T_{coal}(m)>\lfloor \varepsilon^2 N^{2\beta+1}\rfloor t\big)=&\lim_{N\to\infty}\Big(1-\frac{1}{N}\nu_N(0)^2\Big)^{\binom{m}{2}\cdot \lfloor \varepsilon^2 N^{2\beta+1}\rfloor t}
=e^{-t\cdot \binom{m}{2}}.
\end{split}
\ee
This shows that the time of the \emph{first} coalescence is exponentially distributed with the right parameter, independent of the choice of initial distribution $\gamma$ supported on $\{0,...,N^\beta\}.$ In order to see that the subsequent coalescence times are independent of $T_{coal}(m)$ and likewise exponentially distributed, we note that $(A_{T_{coal}(m)+k}^N)_{k\in\N_0}$ has the law of $(A_k^N)_{k\in\N_0}$ started in $m'<m$ distinct blocks, corresponding to independent ancestral lineages $(S_n^{(i)}, 1\leq i\leq m',$ started from some unknown initial distribution $\gamma'$ supported on $\{0,...,N^\beta\}.$ Hence, by the same argument as before,
\emph{all} inter-coalescence times are asymptotically independent and converge to those of Kingman's coalescent. In order to see that there are no multiple or simultaneous (multiple) mergers, we note that due to Lemma \ref{lem:total_variation} for any $\tau_i,$ the probability of a triple merger

\[P_\gamma(X_{\tau_i}=1, Y_{\tau_i}=1, Z_{\tau_i}=1)=\nu_N(0)^3(1+o(1)),\]
which is negligible compared to $\nu_N(0)^2.$ Similarly we see that two simultaneous double mergers are negligible. Standard arguments using exchangeability \cite{MoehleSagitov, Moehle} imply that any multiple or simultaneous (multiple) mergers are negligible. This proves our main result.
\hfill $\Box$

\paragraph*{Proof of Proposition \ref{thm:mrca}.} First of all, observe that at each time step a coalescence happens with probability at most $1/N,$ hence the time to the most recent common ancestor of two lines is a priori bounded by the  time to the most recent common ancestor in the Wright-Fisher model, hence $E[T^N_{MRCA}(2)]\geq N.$ Since clearly $T_{MRCA}(m)\geq T_{MRCA}(2)$ almost surely, we restrict ourselves to the case $m=2.$ Assume first that $\gamma=\delta_0,$ which means that $X_0=Y_0=0.$ Let $J_0:=\inf\{k\geq 1: X_k=N^\beta-1\}.$ By construction, $J_0\sim \mbox{Geo}(\varepsilon).$ Therefore

\be \label{eq5}P_0\big(J_0<T^N_{MRCA}(2)\big)\geq \frac{\varepsilon}{\varepsilon+1/N}\geq 1-\frac{1}{\varepsilon N},\ee
noting that the time until the first coalescence while both processes are in 0 is geometric with parameter $1/N.$ By construction, the dynamics of $X$ consists of excursions away from 0, which have lenght $N^\beta$ each, and after each excursion a period of length $J_i$ in 0, with $J_i$ iid Geo$(\varepsilon).$ Hence coalescence can only happen in (random) time period of the form 
\[\Big\{\sum_{i=0}^l (J_i+N^\beta),...,\sum_{i=0}^l (J_i+N^\beta)+J_{l+1}\Big\}.\]
By the same reasoning as in the derivation of \eqref{eq5},

\be P_0\Big(T^N_{MRCA}(2)\geq \sum_{i=0}^l (J_i+N^\beta)+J_{l+1}\,|\, T^N_{MRCA}(2)>\sum_{i=0}^{l-1} (J_i+N^\beta)+J_{l}\Big)\geq 1-\frac{1}{\varepsilon N}.\ee
Hence the number of excursions away from 0 of one particle before coalescence dominates a geometric random variable with parameter $1/(\varepsilon N),$ and the expected number of excursions away from 0 is at least $\varepsilon N.$ Since each excursion has length at least $N^\beta,$ we get

\be E_0[T^N_{MRCA}(2)]\geq \varepsilon N^{1+\beta}\vee N.\ee
Obviously our argument was independent of the initial distribution $\gamma,$ and hence (i) is proven.

\medskip

Let now $\beta<1/3.$ Let $\tau_0=0,$ and $\tau_k, k\in\N,$ be such that $(\tau_k-\tau_{k-1}),k\in\N$ are independent geometric random variables with parameter $1/N.$ Due to Lemma \ref{lem:change}, $T_{MRCA}(2)\geq \tau_1$ in distribution, and $\tau_1\sim \mbox{Geo}(1/N).$ Hence by \eqref{eq:bound2}, there exists $\delta' >0$ such that 

\[P_\gamma(T^N_{MRCA}(2)<N^{3\beta+\delta'})\to 0\]
as $N\to\infty.$ 
By Lemma \ref{lem:total_variation}, the Markov property, and Lemma \ref{lem:change}, we get
for all $\delta>0$ 
\be \begin{split}
P_{\gamma\otimes \gamma}\big(T^N_{MRCA}(2)\leq \varepsilon^2 N^{1+2\beta-\delta}\,|\, &T^N_{MRCA}(2)>N^{3\beta+\delta'}\big)\\
\leq & P_{\nu_N\otimes \nu_N}\big(T^N_{MRCA}(2)\leq \varepsilon^2 N^{2\beta+1-\delta}-N^{3\beta+\delta'}\big)+o(1)\\
\leq & 1-\big(1-\frac{1}{N}\cdot \nu_N(0)^2\big)^{\varepsilon^2 N^{2\beta+1-\delta}-N^{3\beta+\delta'}}+o(1)\\
\to & 0\, ,
\end{split}
\ee
since 
\[\frac{1}{N}\cdot \nu_N(0)^2\sim \frac{1}{\varepsilon^2 N^{2\beta+1}}=o(\varepsilon^2 N^{2\beta+1-\delta}-N^{3\beta+\delta'})\]
 as $N\to\infty,$ and this finishes the proof. \hfill $\Box$

\section{Simulations}
Since our main results, Theorem~\ref{thm:kingman} and
Corollary~\ref{cor:scaledxptime}, are valid only for $0 < \beta < 1/4,$ it is natural to investigate the coalescence time $\Tm{2}$ defined in \eqref{eq:TMRCAN} for $\beta \ge 1/4$
via simulations.   
The results in Figure~\ref{fig:histbetathird}, which show estimates
(histograms) of the distribution of the scaled time
$\Tm{2}/\left(\varepsilon^2N^{1 + 2\beta} \right)$ when $\beta = 1/3$,
suggest that $\varepsilon^2N^{1 + 2\beta}$ is indeed still the correct scaling.  
The distribution of $\Tm{2}$ fits an exponential distribution with corresponding mean $\overline{T}$ quite well in all cases. The mean $\overline{T}$ should be 1 in all cases. We see that this is not the case for small $\varepsilon,$ but the fit gets better as $\varepsilon$ increases. This can be explained by noting that for small $\varepsilon,$ the scales $\varepsilon^{-2}$ and $N^{1+2\beta}$ may not differ very much, illustrating the fact that the asymptotic results hold only if $\varepsilon$ is independent of $N.$ Similar results and
conclusions hold for $\beta = 1/2$ (Figure~\ref{fig:histbetahalf}).
The estimate of the standard error is close to
the mean in all cases. The simulation results for $\beta = 1/2$ are particularly
interesting since the long ancestral jumps $(\lfloor N^\beta \rfloor)$
are $10^3$ generations when $N = 10^6$, which can be considered a
significant extension of the ancestry of the lines. C code for the
simulations is available upon request.

\medskip
\newpage

\begin{figure}[h!]
\begin{center}
\caption{Histograms of $\Tm{2}/\left(\varepsilon^2N^{1 + 2\beta} \right)$ when $\beta = 1/3$, and $\varepsilon, N$ as shown. The solid fitted line is the density of the exponential with mean the corresponding mean $(\overline{T})$ of the simulated datapoints. Each histogram is normalised to have unit mass one. }%
\vspace{0.2cm}

\label{fig:histbetathird}%
\includegraphics[scale=0.66]{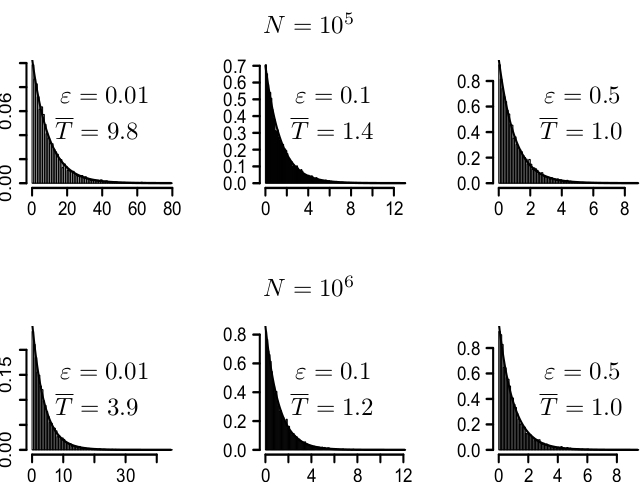}%
\end{center}
\end{figure}

\begin{figure}[h!]
\begin{center}
\caption{Histograms of $\Tm{2}/\left(\varepsilon^2N^{1 + 2\beta} \right)$ when $\beta = 1/2$, and $\varepsilon, N$ as shown. The solid fitted line is the density of the exponential with mean the corresponding mean $(\overline{T})$ of the simulated datapoints. Each histogram is normalised to have unit mass one.  }%
\vspace{0.2cm}
\label{fig:histbetahalf}%
\includegraphics[scale=0.66]{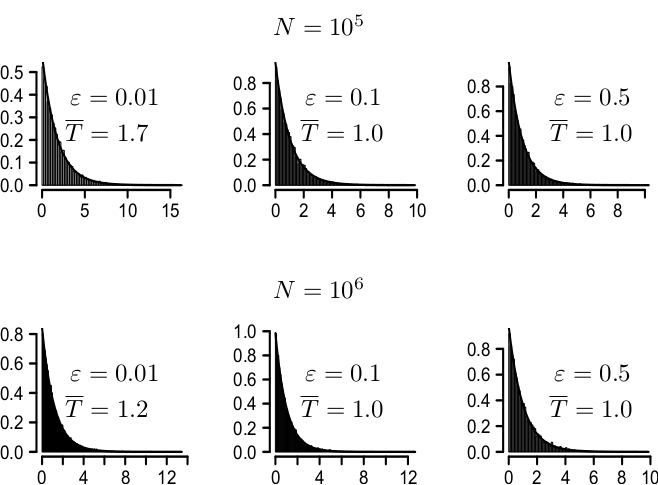}%
\end{center}
\end{figure}

\newpage

\paragraph*{Acknowledgements.} JB, BE and NK acknowledge support by the DFG SPP 1590 ``Probabilistic structures in evolution''. AGC is supported by the DFG RTG 1845, the Berlin Mathematical School (BMS), and the Mexican Council of Science in collaboration with the German Academic Exchange Service (DAAD). The authors wish to thank Julien Berestycki and Dario Span\`o for interesting discussions.

\end{document}